\documentclass[10pt]{amsart}  
\usepackage{amsmath,amssymb,amsfonts,color,hyperref,graphicx}

\title{A simplex algorithm for rational cp-factorization}

\author{Mathieu Dutour Sikiri\'c} 
\address{M. Dutour Sikiri\'c, Rudjer Boskovi\'c Institute, Bijenicka 54, 10000 Zagreb, Croatia}
\email{mdsikir@irb.hr}

\author{Achill Sch\"urmann} 
\address{A. Sch\"urmann, Institute of Mathematics, University of Rostock, 18051 Rostock, Germany}
\email{achill.schuermann@uni-rostock.de}

\author{Frank Vallentin} 
\address{F. Vallentin, Mathematisches Institut, Universit\"at zu K\"oln, Weyertal 86--90, 50931 K\"oln, Germany}
\email{frank.vallentin@uni-koeln.de}

 

\date{January 25, 2020}

\newtheorem{defi}{Definition}[section]
\newtheorem{definition}[defi]{Definition}

\newtheorem{theorem}[defi]{Theorem}

\newtheorem{lemma}[defi]{Lemma}
\newtheorem{conj}[defi]{Conjecture}

\newcommand{\Z}{{\mathbb{Z}}} 
\newcommand{\Q}{{\mathbb{Q}}}
\newcommand{\R}{{\mathbb{R}}} 
\newcommand{\MV}{{\mathcal{V}}} 

\DeclareMathOperator{\minC}{min_{\mathcal{COP}}}
\DeclareMathOperator{\MinC}{Min_{\mathcal{COP}}}
\DeclareMathOperator{\Trace}{Trace}
\DeclareMathOperator{\cone}{cone}
\DeclareMathOperator{\conv}{conv}

\DeclareMathOperator{\bd}{bd}
\DeclareMathOperator{\interior}{int}
\DeclareMathOperator{\Id}{Id}
\DeclareMathOperator{\rank}{rank}

\newcommand{\sn}{{\mathcal S}^{n}} 
\newcommand{\sngeo}{{\mathcal S}^{n}_{\geq 0}} 
\newcommand{\sngo}{{\mathcal S}^{n}_{>0}}
 
\newcommand{\cop}{{\mathcal{COP}}} 
\newcommand{\copn}{{\mathcal{COP}}^{n}} 
\newcommand{\cpn}{{\mathcal{CP}}^{n}} 
\newcommand{\cptwo}{{\mathcal{CP}}^{2}} 
\newcommand{\cpnrat}{\widetilde{\mathcal{CP}}^{n}}
\newcommand{\cprat}{\widetilde{\mathcal{CP}}}

\newcounter{alg}

\newenvironment{bigalg}{\medskip%
                        \begin{figure}[htbp]%
                        \refstepcounter{alg}%
                        \begin{center}}%
                       {\end{center}\end{figure}\medskip}

\begin{document}

\begin{abstract} 
  In this paper we provide an algorithm, similar to the simplex
  algorithm, which determines a rational cp-factorization of a given
  matrix, whenever the matrix allows such a factorization. This
  algorithm can be used to show that every integral completely positive
  $2 \times 2$ matrix has an integral cp-factorization.
\end{abstract}

\maketitle

\markboth{M.~Dutour Sikiri\'c, A. Sch\"urmann, and F.~Vallentin}{A
  simplex algorithm for rational cp-factorization}

\section{Introduction}

Copositive programming gives a common framework to formulate many
difficult optimization problems as convex conic ones.  In fact, many
$\textrm{NP}$-hard problems are known to have such reformulations~(see
for example the surveys \cite{Bomze2012a,Duer2010a}).  All the
difficulty of these problems appears to be ``converted'' into the
difficulty of understanding the cone of copositive matrices $\copn$
which consists of all symmetric $n \times n$ matrices
$B \in \mathcal{S}^n$ with $x^{\sf T} B x \geq 0$ for all
$x \in \mathbb{R}^n_{\geq 0}$. Its dual cone is the cone
\[
\begin{split}
\cpn
\; = \; & \cone\{xx^{\sf T} : x \in \mathbb{R}^n_{\geq 0}\}\\
 \; = \; & \left\{\sum_{i=1}^m \alpha_i x_i x_i^{\sf T} : m \in \mathbb{N},
\alpha_i \in \mathbb{R}_{\geq 0}, x_i \in \mathbb{R}^n_{\geq
  0}, i = 1, \ldots, m\right\}
\end{split}
\] 
of completely positive $n\times n$ matrices. Therefore, it seems no
surprise that many basic questions about this cone are still open and
appear to be very difficult.

One important problem is to find an algorithmic test deciding whether
or not a given symmetric matrix $A$ is completely positive. If
possible one would like to obtain a certificate for either
$A \in \cpn$ or $A \not \in \cpn$. 
Dickinson and Gijben
\cite{Dickinson2014a} showed that this (strong) membership problem is
$\textrm{NP}$-hard.

In terms of the definitions the
most natural certificate for $A \in \cpn$ is giving a cp-factorization
\begin{equation} 
\label{eqn:cp-factorization}
  A = \sum_{i=1}^m \alpha_i x_i x_i^{\sf T} \quad \text{with}
\quad
  m \in \mathbb{N},
\alpha_i \in \mathbb{R}_{\geq 0}, \; x_i \in \mathbb{R}^n_{\geq
  0}, \; i = 1, \ldots, m.
\end{equation} 
For $A \not\in \cpn$ it is natural to give a separating hyperplane defined by
a matrix $B \in \copn$ so that the inner product of $A$ and $B$ satisfies
$\langle B, A \rangle < 0$. 

From the algorithmic side, different ideas have been proposed.  One
can divide the relevant literature according to two complementary
approaches:
\begin{enumerate}
\item Numerical methods which are practical but ``only'' can
find approximate cp-factorizations. The papers by  Jarre and Schmallowsky
\cite{Jarre2009a}, Nie \cite{Nie2014a}, Sponsel and D\"ur
\cite{Sponsel2014a}, Elser~\cite{Elser2017a}, Groetzner and D\"ur
\cite{Groetzner2018a} fall into this category.
\item Theoretical methods which can compute exact cp-factorizations in
finitely many algorithmic steps. The factorization method of
Anstreicher, Burer, and Dickinson \cite[Section 3.3]{Dickinson2013a}
uses the ellipsoid method and works for all matrices which have a
rational cp-factorization and lie in the interior of the cone $\cpn$.
Berman and Rothblum \cite{Berman2006a} use quantifier elimination for
first order formulae over the reals to compute the $\mathcal{CP}$-rank of a
given matrix, that is, the minimum number $m$ of vectors used in a 
cp-factorization \eqref{eqn:cp-factorization}.
\end{enumerate}

In this paper, in Section~\ref{sec:algorithm}, we describe a new
procedure that is based on pivoting like the simplex algorithm. To
define the pivoting we apply the notion of the copositive minimum
which we introduce in Section~\ref{sec:copmin}. Our algorithm
(Algorithm~\ref{alg:new-voronoi-algorithm}) works for all matrices in
the rational cone
\[
\begin{split}
\cpnrat \; = \; & \cone_{\mathbb{Q}}\{xx^{\sf T} : x \in \mathbb{Q}^n_{\geq
  0}\} \\
 \; = \; & \left\{\sum_{i=1}^m \alpha_i x_i x_i^{\sf T} : m \in \mathbb{N},
\alpha_i \in \mathbb{Q}_{\geq 0}, x_i \in \mathbb{Q}^n_{\geq
  0}, i = 1, \ldots, m\right\}.
\end{split}
\]
Moreover, we conjecture that a variant of our algorithm
(Procedure~\ref{alg:new-extended-procedure}) always computes
separating hyperplanes, if the input matrix is not completely
positive. Overall, our procedure works for matrices with coefficients
in any computable subfield $F$ of the real numbers, in that case the
coefficients $\alpha_i$ of the formula above belong to $F_{\geq 0}$ and
the whole algorithmic procedure works similarly as the rational case
that we consider in this paper.

Our algorithm uses rational numbers only if the input matrix is
rational and so allows in principle exact computations.  As a
consequence, to the best of our knowledge, our algorithm is currently
the only one that can find a rational cp-factorization whenever it
exists.  In \cite{Dickinson2013a} a similar result was obtained, but
restricted to matrices in the interior of $\cpn$. A related question
is if every rational completely positive matrix has a rational
cp-factorization, see the survey \cite{Shaked2018a}. Generally we do not know
but from the results in~\cite{Dickinson2013a} and \cite{Dutour2017a} it follows
that this is true for matrices in the interior of $\cpn$.

If the input matrix $A$ is integral, one can also ask if it admits an
integral cp-factorization, i.e.\ a cp-factorization of the form
$A = \sum_{i = 1}^m x_i x_i^{\sf T}$ with
$x_i \in \mathbb{Z}_{\geq 0}^n$ for all $i = 1, \ldots, m$. For $n \geq 3$
it is known that there are integral matrices $A \in \cpn$ which do not
have an integral cp-factorization, see~\cite[Theorem
6.4]{Berman2018a}. For $n = 2$ it was conjectured by Berman and
Shaked-Monderer \cite[Conjecture 6.13]{Berman2018a} that every
integral matrix $A \in \cptwo$ possesses an integral
cp-factorization. This conjecture was recently proved by Laffey and
\v{S}imgoc \cite{Laffey2018a}. In Section~\ref{sec:2dexample} we show
that our simplex algorithm can be used to give a short, alternative
proof of this result.

In Section~\ref{sec:experiments} we describe how an implementation of
our algorithm performs on some examples.

\section{The copositive minimum and copositive perfect matrices}
\label{sec:copmin}

\subsection{Copositive minimum}

By $\sn$ we denote the Euclidean vector space of symmetric $n\times n$
matrices with inner product
$\langle A, B \rangle = \Trace(AB) = \sum_{i,j=1}^n
A_{ij}B_{ij}$. With respect to this inner product we have the
following duality relations between the cone of copositive matrices
and the cone of completely positive matrices
\[
\copn = (\cpn)^* = \{B \in \mathcal{S}^n : \langle A, B \rangle \geq 0 \text{ for all } A \in \cpn \},
\]
and
\[
\cpn = (\copn)^*.
\]
So, in order to show that a given symmetric matrix $A$ is not completely
positive, it suffices to find a copositive matrix $B \in \copn$ with
$\langle B, A\rangle < 0$. We call $B$ a \emph{separating witness} for
$A \not\in \cpn$ in this case, because the linear hyperplane orthogonal to $B$
separates $A$ and $\cpn$.

Using the notation $B[x]$ for
$x^{\sf T} B x = \langle B, xx^{\sf T} \rangle$, we obtain
\[
\copn = \{B \in \mathcal{S}^n :  B[x] \geq 0 \text{ for all } x\in \mathbb{R}^n_{\geq 0}\}.
\]
Obviously, the cone of positive semidefinite matrices $\sngeo$, whose interior
is the open cone of positive definite matrices $\sngo$, lies between the
completely positive cone and the copositive cone: 
$\cpn \subseteq \sngeo \subseteq \cop^n$.

\begin{definition}
For a symmetric matrix~$B\in \sn$ we define the
\emph{copositive minimum} as 
\[
\minC (B) = \inf \left\{ B[v]  :  v \in \Z^n_{\geq 0}\setminus\{0\} \right\},
\]
and we denote the set of vectors attaining it by
\[
\MinC (B) = \left\{ v \in \Z^n_{\geq 0} :  B[v] = \minC(B) \right\}.
\]
\end{definition}

The following proposition shows that matrices in the interior of the
cone of copositive matrices attain their copositive minimum.

\begin{lemma}
\label{lem:mincop}
  Let $B$ be a matrix in the interior of the cone of copositive
  matrices. Then the copositive minimum of $B$ is strictly positive
  and it is attained by only finitely many vectors.
\end{lemma}

\begin{proof}
  Since $B$ is copositive, we have the inequality $\minC(B) \geq 0$.
  Suppose that $\minC(B) = 0$.  Then there is a sequence
  $v_i \in \Z_{\geq 0}^n \setminus \{0\}$ of pairwise distinct lattice
  vectors such that $B[v_i]$ tends to zero when $i$ tends to infinity.
  From the sequence $v_i$ we construct a new sequence $u_i$ of vectors
  on the unit sphere $S^{n-1}$ by setting $v_i = \|v_i\| u_i$. The sequence
  $u_i$ belongs to the compact set $\R^n_{\geq 0} \cap S^{n-1}$.
  Thus, by taking a subsequence if necessary, we may assume that $u_i$
  converges to a point $u\in \R^n_{\geq 0} \cap S^{n-1}$.  The
  sequence of norms $\|v_i\|$ tends to infinity since the set of
  lattice vectors of bounded norm is finite. Thus we get
\[
0 = \lim_{i \to \infty} B[v_i] = \lim_{i \to \infty} \|v_i\|^2 B[u_i],
\]
which implies that $B[u] = 0$, contradicting our assumption 
$B \in \interior(\copn)$. Hence, $\minC(B) > 0$.

By the same argument one can show that $\MinC(B)$ only contains
finitely many vectors.
\end{proof}

\subsection{A locally finite polyhedron}
\label{ssec:locallyfinite}

In our previous paper~\cite{Dutour2017a} and in this paper the set
\[
\mathcal{R} = \{B \in \sn : B[v] \geq 1 \text{ for
  all } v \in \mathbb{Z}^n_{\geq 0} \setminus \{0\} \}
\]
plays a central role\footnote{We use the letter $\mathcal{R}$ here because
  Ryshkov used a similar construction in the study of lattice sphere
  packings, see for example~\cite[Chapter 3]{Schuermann2009a}}. The set $\mathcal{R}$ is a locally finite
  polyhedron, meaning that every intersection of $\mathcal{R}$ with a
  polytope is a polytope itself. In \cite[Lemma 2.3]{Dutour2017a} we showed that
$\mathcal{R}$ is contained in the interior of the cone of copositive
matrices. Thus, we can rewrite $\mathcal{R}$ as
\begin{equation} 
\label{eqn:ryshkov-set}
\mathcal{R} = \{B \in \sn : \minC(B) \geq 1 \}.
\end{equation}
Note also that \cite[Lemma 2.3]{Dutour2017a} together with
Lemma~\ref{lem:mincop} implies
\begin{equation}
\label{eqn:coneR}
\cone \mathcal{R} \setminus \{0\} = \interior \cop^n.
\end{equation}

The following theorem gives a tight outer
approximation of the cone of completely positive matrices in terms of
the boundary structure (its $1$-skeleton to be precise) of the convex
set~$\mathcal{R}$. Similarly, Y\i{}ld\i{}r\i{}m \cite{Yildirim2012a}
discusses uniform polyhedral approximations of the the cone of
copositive matrices.

\begin{theorem}
We have
\label{thm:outer-approx}
\begin{equation}
\label{eq:outer-approx}
\begin{split}
\cpn = 
\{ Q\in \sn : \langle Q, B \rangle \geq 0 & \text{ for all vertices
  and for all generators}\\
 & \text{of extreme rays $B$ of $\mathcal{R}$} \}.
\end{split}
\end{equation}
\end{theorem}

\begin{proof}
We have 
\[
\cpn = (\copn)^* = (\interior(\copn))^* = (\cone \mathcal{R})^*, 
\]
where the identity $K^* = (\interior(K))^*$ is generally true for full
dimensional convex cones and the last identity is~\eqref{eqn:coneR}.
Since $\mathcal{R}$ is a locally finite polyhedron, $(\cone
\mathcal{R})^*$ is equal to the right hand side of~\eqref{eq:outer-approx}.
\end{proof}

\subsection{A linear program for finding a rational cp-factorization}
\label{ssec:linearprogram}

In \cite[Lemma 2.4]{Dutour2017a} we showed that for
$A\in \interior(\cpn)$ and all sufficiently large $\lambda>0$ the set
\[
\mathcal{P}(A,\lambda) = \{B\in \mathcal{R} : \langle A, B \rangle
\leq \lambda\}
\]
is a  full-dimensional polytope.

In principle (cf.~\cite[Proof of Theorem 1.1]{Dutour2017a}), this gives a
way to compute a cp-factorization for a given matrix
$A\in \interior(\cpn)$ by solving the linear program
\begin{equation}
\label{eqn:linear_program}
\min\left\{\langle A, B \rangle : B \in \mathcal{P}(A,\lambda)\right\}:
\end{equation}
This is because the minimum is attained at a vertex $B^*$ of
$\mathcal{P}(A,\lambda)$. Hence, due to the minimality of $\langle
A, B^* \rangle$, the matrix $A$ is contained in the \emph{(inner) normal cone}
\begin{equation}
\label{eq:mvb}
\MV(B^*) =  \cone \left\{vv^{\sf T}  : v \in \MinC B^* \right\}
\end{equation}
of $\mathcal{R}$ at $B^*$.  For a rational matrix $A\in \interior(\cpn)$
we obtain a \emph{rational cp-factorization} in this way, that is, a
decomposition of the form
\begin{equation}
\label{eq:rationalcpfactorization}
A = \sum_{i=1}^m \alpha_i v_i v_i^{\sf T} \quad \text{with } \alpha_i
\in \Q_{\geq 0 }\text{ and } v_i \in \Z^n_{\geq 0}, 
\quad \text{for } i = 1, \ldots, m.
\end{equation}
To find this factorization, we apply Carath\'eodory's theorem (see for
example \cite[Corollary 7.1i]{Schrijver1986a}) and choose a subset
$v_1, \ldots, v_m \in \MinC B^*$ such that $v_iv_i^{\sf T}$ are
linearly independent and
$A \in \cone\{v_i v_i^{\sf T} : i = 1, \ldots, m\}$. So we can find
unique non-negative rational coefficients $\alpha_1, \ldots, \alpha_m$
giving the rational cp-factorization
\eqref{eq:rationalcpfactorization}.

However, for solving the linear program~\eqref{eqn:linear_program} one
needs an explicit finite algorithmic description of the set
$\mathcal{P}(A, \lambda)$, for example by a finite list of linear
inequalities. The proof of the polyhedrality of
$\mathcal{P}(A, \lambda)$ in \cite[Lemma 2.4]{Dutour2017a} relies on
an indirect compactness argument (similar to the one in the proof of
Lemma~\ref{lem:mincop}) which does not yield such an explicit
algorithmic description. In the remainder of this paper we are
therefore concerned with finding a finite list of linear inequalities.

\subsection{Copositive perfect matrices}

In the next step we characterize the vertices of $\mathcal{R}$. The
following definitions and the algorithm in the following section are
inspired by Voronoi's classical algorithm for the classification of
perfect positive definite quadratic forms. These can for instance be
used to classify all locally densest lattice sphere packings (see for
example \cite{Martinet2003a} or \cite{Schuermann2009a}).  In
\eqref{eq:mvb} we use the letter $\MV$ to denote the normal cone of a
vertex, as it is a generalization of the \emph{Voronoi cone} used in
the classical setting. In fact, our generalization of Voronoi's work
can be viewed as an example of a broader framework described by
Opgenorth~\cite{Opgenorth2001}.  In analogy with Voronoi's theory for
positive definite quadratic forms we define the notion of perfectness
for copositive matrices:

\begin{definition}
  A copositive matrix $B \in \interior(\copn)$ is called
  \emph{$\cop$-perfect} if it is uniquely determined by its copositive
  minimum $\minC B$ and the set $\MinC B$ attaining it.
\end{definition}

In other words, $B \in \interior(\copn)$ is $\cop$-perfect if and only
if it is the unique solution $X$ of the system of linear equations
\[
\langle X, vv^{\sf T}\rangle = \minC B, \text{ for all } v \in \MinC B.
\]
Hence, $\cop$-perfect matrices are, up to scaling, exactly the vertices of $\mathcal{R}$.

\begin{lemma}
  $\cop$-perfect matrices exist in all dimensions (dimension $n = 1$ being
  trivial): For dimension $n \geq 2$ the following matrix
\begin{equation}
\label{eqn:initial_perfect_form}
Q_{\mathsf{A}_n}
=
\begin{pmatrix}
\phantom{-}2 & -1 & 0 & \ldots & 0 \\
-1 & \phantom{-}2 & \ddots &  \ddots & \vdots \\
0  & \ddots & \ddots & \ddots & 0 \\
 \vdots &    \ddots &  \ddots &  \phantom{-}2 & -1 \\
 0 & \ldots & 0   &  -1 & \phantom{-}2 \\
\end{pmatrix}
\end{equation}
is $\cop$-perfect; $\frac{1}{2} Q_{\mathsf{A}_n}$
is a vertex of $\mathcal{R}$.
\end{lemma}

The matrix $Q_{\mathsf{A}_n}$ is also known as a Gram matrix of the
root lattice $\mathsf{A}_n$, a very important lattice, for instance in the theory of
sphere packings (see for example~\cite{Conway1988a}).

\begin{proof}
The matrix $Q_{\mathsf{A}_n}$ is positive definite since 
\[
Q_{\mathsf{A}_n}[x] = x_1^2 + \sum_{i=1}^{n-1} (x_i - x_{i+1})^2 + x_n^2 
\]
is a sum of squares and $Q_{\mathsf{A}_n}[x] = 0$ if and only
if $x = 0$. Thus $Q_{\mathsf{A}_n}$ lies in the interior of the copositive
cone. Furthermore,
\[
\minC Q_{\mathsf{A}_n} = 2
\quad
\text{with}
\quad
\MinC Q_{\mathsf{A}_n} = \left\{\sum_{i=j}^k e_i : 1 \leq j \leq k
  \leq n\right\},
\]
where $e_i$ is the $i$-th standard unit basis vector of
$\mathbb{R}^n$. Thus, the ${n+1 \choose 2}$ vectors
attaining the copositive minimum have a continued sequence of $1$s in
their coordinates and $0$s otherwise. Now it is easy to see that the
rank-$1$-matrices
\[
\left(\sum_{i=j}^k e_i\right) \left(\sum_{i=j}^k e_i\right)^{\sf T},
\text{ where }  1 \leq j \leq k \leq n,
\]
are linearly independent and span the space of symmetric matrices
which shows that $Q_{\mathsf{A}_n}$ is $\cop$-perfect.
\end{proof}

\section{Algorithms}
\label{sec:algorithm}

In this section we show how one can solve the linear
program~\eqref{eqn:linear_program}.  Our algorithm is similar to the
simplex algorithm for linear programming. It walks along a path of
subsequently constructed $\cop$-perfect matrices, which are vertices
of the polyhedral set $\mathcal{R}$ that are connected by edges of~$\mathcal{R}$.

We start with a simple version assuming that the input matrix lies in
$\cpnrat$. Of course, this assumption can usually not been easily
checked beforehand and the rational cp-factorization is only given as
the output of the algorithm. In this sense, the algorithm gets the
promise that the input matrix possesses a rational cp-factorization. In
theoretical computer science promise problems are common; for
practical purposes we propose an extended procedure at the end of this
section, see Procedure~\ref{alg:new-extended-procedure}.

\begin{bigalg}   
\label{alg:new-voronoi-algorithm}
\begin{minipage}{14cm}
\begin{flushleft}
\textbf{Input:} $A\in \cpnrat$\\[0.2cm] 
\textbf{Output:} Rational cp-factorization of $A$.

\begin{enumerate}
\item[1.] Choose an initial $\cop$-perfect matrix  $P \in
  \mathcal{R}$; initialize $\mathcal{V}(P)$.\\

\item[2.] {\tt while} $A \not\in \mathcal{V}(P)$
  
\item[] \hspace*{0.5cm}(a) Determine a generator $R$ of an extreme ray of
  $(\MV(P))^{\ast}$ with $\langle  A,R \rangle < 0$.\\

\item[] \hspace*{0.5cm}(b) Use
  Algorithm~\ref{alg:determination-of-C-voronoi-neighbors} to
  determine the contiguous $\cop$-perfect matrix\\
  \hspace*{1cm} $N \leftarrow P + \lambda R$
  with $\lambda > 0$ and $\minC (N) = 1$. Compute $\mathcal{V}(N)$.

  \item[] \hspace*{0.5cm}(c) $P \leftarrow N$

\item[3.] Determine $\alpha_1, \ldots, \alpha_m \in \mathbb{Q}_{\geq 0}$ with
  $A = \sum_{i=1}^m \alpha_i v_iv_i^{\sf T}$ and output this rational
  cp-factorization.
\end{enumerate}
\end{flushleft}
\end{minipage}
\\[1ex]
{\sc Algorithm \arabic{alg}.} Algorithm to find
a rational cp-factorization 
\end{bigalg}

\subsection{Description and analysis of the algorithm}

In the following we describe the steps of
Algorithm~\ref{alg:new-voronoi-algorithm} in more detail:

\medskip

In \textbf{Step~1}, we can choose for instance the initial vertex
$P=\frac{1}{2}Q_{\mathsf{A}_n}$ of~${\mathcal R}$ with $Q_{\mathsf{A}_n}$ as
in~\eqref{eqn:initial_perfect_form}.  Then the algorithm subsequently
constructs vertices of ${\mathcal R}$.

\medskip

In \textbf{Step~2} we determine whether $A$ lies in the polyhedral
cone $\MV(P)$. For this we consider all $v\in \MinC(P)$ giving
generators $vv^{\sf T}$ of the polyhedral cone $\MV(P)$, respectively
defining linear inequalities of the dual cone $(\MV(P))^*$.  Testing
$A\in \MV(P)$ can then be done by solving an auxiliary linear program
\begin{equation}
\label{eq:auxiliarylp}
\min\left\{\langle A,Q \rangle : Q \in (\MV(P))^*\right\}.
\end{equation}
The minimum equals $0$ if and only if $A$ lies in $\MV(P)$.  If
$A \in \MV(P)$, then we can find non-negative coefficients
$\lambda_v$, with $v \in \MinC(P)$, to get a cp-factorization
\[
A = \sum_{v \in \MinC(P)} \lambda_v vv^{\sf T}.
\]
Using (an algorithmic version of) Carath\'eodory's theorem we can
choose in \textbf{Step~3} a subset
$\{v_1, \ldots, v_m\} \subseteq \MinC(P)$ so that we get a rational
cp-factorization $A = \sum_{i = 1}^m \alpha_i v_iv_i^{\sf T}$ with
non-negative rational numbers $\alpha_i$; see 
Section~\ref{ssec:linearprogram}.

\medskip

If the minimum of the auxiliary linear program \eqref{eq:auxiliarylp}
is negative we can find in \textbf{Step~2(a)} a generator $R$ of an
extreme ray of $(\MV(P))^{\ast}$ with $\langle A,R \rangle < 0$.
Here, several choices of $R$ with $\langle A,R \rangle < 0$ may be
possible and the performance depends on the choices made in this
``pivot step''.  A good heuristic for a ``pivot rule'' seems to be the
choice of~$R$ with $\langle A,R / \|R\| \rangle$ minimal, where
$\|R\|^2 = \langle R, R\rangle $.  Also a random choice of $R$ among
the extreme rays of $(\MV(P))^{\ast}$ with $\langle A,R \rangle < 0$
seems to perform quite well.  When choosing the right pivots~$R$ in
Step~2(a) Algorithm~\ref{alg:new-voronoi-algorithm} always terminates,
as shown by Theorem~\ref{thm:alg-works-rational-closure} below.

\medskip

In \textbf{Step~2(b)}
Algorithm~\ref{alg:determination-of-C-voronoi-neighbors}
(see Section~\ref{ssec:contiguous} is used to
determine a new {\em contiguous $\cop$-perfect matrix} $N$ of $P$ 
in direction of $R\not \in \copn$, that is, a contiguous vertex of
$P$ on $\mathcal{R}$, connected via an edge in direction $R$.  Note
that such a vertex exists (and $\mathcal{R}$ is not unbounded in the
direction of $R$) under the assumption $R \not \in \copn$, 
because $\mathcal{R} \subseteq \interior(\copn)$, 
see \cite[Lemma 2.3]{Dutour2017a}. 
We can exclude $R \in \copn$ here, since together with $\langle  A,R \rangle < 0$
it would contradict the promise $A\in \cpnrat$ on the input.
Note also that as a byproduct of
Algorithm~\ref{alg:determination-of-C-voronoi-neighbors} we compute
generators of the cone $\mathcal{V}(N)$.

\medskip

Finally, we observe that since $\langle A,R \rangle < 0$, we have
$\langle A,N \rangle < \langle A,P \rangle$ in each iteration
(Step~2) of the algorithm.

\medskip

The following theorem shows that we can set up an algorithm for the
promise problem.

\begin{theorem} 
\label{thm:alg-works-rational-closure}
For $A\in \cpnrat$, Algorithm~\ref{alg:new-voronoi-algorithm} with
suitable choices in Step~2(a) ends after finitely many iterations
giving a rational cp-factorization of $A$.

In particular, with breadth-first-search added to
Algorithm~\ref{alg:new-voronoi-algorithm}  we can guarantee finite
termination (but this of course would be far less efficient).
\end{theorem}

\begin{proof}
  For $A\in \interior(\cpn) \cap \cpnrat$ the assertion follows
  from Lemma~2.4 in \cite{Dutour2017a}.

  So let us assume $A\in \bd \cpn \cap \cpnrat$. Note that $\cpnrat$
  is tessellated into cones $\MV(P)$ of $\cop$-perfect matrices~$P$.
In fact, the convex hull of $D = \{xx^{\sf T} : x \in \Z^n_{\geq 0}, x\not=0 \}$ 
is a locally finite polyhedral set whose facets are in one-to-one correspondence 
with the $\cop$-perfect matrices (see \cite{Opgenorth2001}). 
For any $A\in \cpnrat$, the ray $\{ \lambda A : \lambda\geq 0\}$ meets 
(at least) one facet $F$ of $\conv D$ and 
$A$ is in $\MV(P) = \cone F$ of the corresponding $\cop$-perfect matrix $P$.

  Let $\{ R_1, R_2, \ldots \}$ be a possible sequence of generators of
  rays constructed in Step~2(a) of
  Algorithm~\ref{alg:new-voronoi-algorithm}.  For all of these generators,
  the inequality $\langle A, R_i\rangle < 0$ holds.  For $k$ such generators, the
  conditions $\langle Q, R_i\rangle < 0$ for $i=1,\ldots , k$ are not
  only satisfied for $Q=A$, but also for all $Q$ in an
  $\varepsilon$-neighborhood of $A$ (with a suitable $\varepsilon$ depending
  on $k$).  For any $k$, this neighborhood also contains points of
  $\interior(\MV(P)) \subseteq \interior(\cpn)$.  For these interior
  points~$Q$, however, Algorithm~\ref{alg:new-voronoi-algorithm}
  finishes after at most finitely many steps (when checking for
  $Q \in \MV(P)$ in Step~2).  Thus, for some finite number of
  suitable choices in Step~2(a), the algorithm also ends for~$A$.
\end{proof}

\subsection{Computing contiguous $\cop$-perfect matrices}
\label{ssec:contiguous}

Our algorithm for computing contiguous $\cop$-perfect matrices is
inspired by a corresponding algorithm for computing contiguous perfect
positive definite quadratic forms which is a subroutine in Voronoi's
classical algorithm. The following algorithm is similar to
\cite[Section 6, Erratum to algorithm of Section 2.3]{Dutour2007a} and
\cite{Woerden2018a}.

\begin{bigalg}   
\label{alg:determination-of-C-voronoi-neighbors}
\begin{flushleft}
\smallskip
\textbf{Input:} $\cop$-perfect matrix $P\in \mathcal{R}$, generator $R
\not \in \copn$ of extreme ray of the polyhedral cone $(\MV(P))^{\ast}$\\
\textbf{Output:} Contiguous vertex $N$
of $P$ on $\mathcal{R}$, connected via an edge in direction $R$,
i.e.\ 
\[
N = P + \lambda R \text{ with } \lambda>0,\; \minC (N) = 1,\; \MinC(N) \not \subseteq \MinC (P).
\]

\smallskip

\begin{enumerate}
\item[1.] $(l, u)  \leftarrow  (0,1)$\\

\item[2.] {\tt while} $P + u R \not\in \interior(\copn)$ or $\minC(P + u R) = 1$ {\tt do}\\
\hspace{2ex} {\tt if} $P + u R \not\in \interior(\copn)$ {\tt then} $u  \leftarrow  (l + u)/2$\\
\hspace{2ex} {\tt else} $(l,u)  \leftarrow  (u, 2u)$\\

\item[3.] $S  \leftarrow  \left\{ v \in \Z^n_{\geq 0} : (P+ u R)[v] < 1 \right\}$

\item[4.]  $\lambda \leftarrow \min \left\{ (1-P[v]) / R[v] :  
                           v \in S \right\},   \;
                   N \leftarrow P+\lambda R$

\end{enumerate}
\end{flushleft}
{\sc Algorithm \arabic{alg}.} 
Determination of a contiguous $\cop$-perfect matrix.
\end{bigalg}

Computationally the most involved parts of
Algorithm~\ref{alg:determination-of-C-voronoi-neighbors} are checking
if a matrix lies in the interior of the cone of copositive matrices,
and if so, computing its copositive minimum $\minC$ and all vectors
$\MinC$ attaining it. We discuss these tasks in
Sections~\ref{ssec:checking_copositivity} and~\ref{ssec:computing_minC}.

In the {\tt while} loop ({\bf Step 2}) 
of Algorithm~\ref{alg:determination-of-C-voronoi-neighbors}, lower and
upper bounds $l$ and $u$ for the desired value $\lambda$ are computed,
such that $P+ l R$ and $P + u R$ are lying in $\interior(\copn)$
satisfying
\[
\minC(P+ l R)=\minC(P) > \minC(P + u R).
\]
In other words, $P + lR$ lies on the edge $[P,N] \subseteq \mathcal{R}$, but $P
+ uR$ lies outside of $\mathcal{R}$.

The set $S$ in {\bf Step 3} contains all vectors
$v \in \mathbb{Z}^n_{\geq 0}$ defining a separating hyperplane
$\{X \in \mathcal{S}^n : \langle X, vv^{\sf T} \rangle = 1\}$,
separating $\mathcal{R}$ and $P + uR$.

If $v \in S$, then $R[v] < 0$ and
\[
(P+\lambda R)[v] = P[v] + \min_{w \in S} \left(\frac{1-P[w]}{R[w]}\right) R[v] \geq 1.
\]
If $v \not\in S$ and $R[v] \geq 0$, then clearly $(P + \lambda R)[v]
\geq 1$, since $\lambda\geq 0$. Finally, if $v \not \in S$ and $R[v] < 0$, then since
$\lambda \leq u$, we have
\[
(P+\lambda R)[v] \geq (P+u R)[v] \geq 1.
\]

Therefore, the choice of $\lambda$ in \textbf{Step~4} guarantees that
$P + \lambda R$ is the contiguous $\mathcal{COP}$-perfect matrix of
$P$.  We have $\minC(P + \lambda R) = 1$ but
$\MinC(P + \lambda R) \not\subseteq \MinC(P)$.

In practice the set $S$ in Step~3 is maybe too big for a complete
enumeration. In this case partial enumerations may help to pick
successively smaller $u$'s first, which are not necessarily equal to the
desired $\lambda$; see \cite{Woerden2018a}.

\subsection{Checking copositivity} 
\label{ssec:checking_copositivity}

From a complexity point of view, checking whether or not a given
symmetric matrix is copositive is known to be co-NP-complete by a
result of Murty and Kabadi~\cite{MurtyKabadi1987}.

Nevertheless, in our algorithms we need to check whether or not a
given symmetric matrix lies in the cone of copositive matrices
(Step~2(c) of Procedure~\ref{alg:new-extended-procedure}) or in its
interior (Step~2 of
Algorithm~\ref{alg:determination-of-C-voronoi-neighbors}).  This can
be checked by the following recursive characterization of Gaddum
\cite[Theorem 3.1~and~3.2]{Gaddum1958a}, which of course is not
computable in polynomial time: By
\[
\Delta = \left\{x \in \mathbb{R}^n : x \geq 0, \sum_{i=1}^n x_i = 1\right\}
\]
we denote the $(n-1)$-dimensional standard simplex in dimension~$n$. 
A matrix $B \in \mathcal{S}_n$ lies in $\copn$ (in $\interior(\copn)$) 
if and only if
every of its principal minors of size $(n-1)\times(n-1)$ lies in $\copn$
(in $\interior(\mathcal{COP}_{n-1})$) and the value
\begin{equation} \label{v-in-cop-characterization}
v = \max_{x \in \Delta} \min_{y \in \Delta} x^{\sf T} B y
= \min_{y \in \Delta} \max_{x \in \Delta} x^{\sf T} B y.
\end{equation}
of the two-player game with payoff matrix $B$ is non-negative
(strictly positive).

One can compute the value of $v$ in \eqref{v-in-cop-characterization}
by a linear program:
\[
v = \max\{\lambda : \lambda \in \mathbb{R}, y \in \Delta, By \geq \lambda e\},
\]
where $e = (1, \ldots, 1)^{\sf T}$ is the all-ones vector.

\subsection{Computing the copositive minimum}
\label{ssec:computing_minC}

Once we know that a given symmetric matrix $B$ lies in the interior of
the copositive cone (i.e.\ after Step~2 of
Algorithm~\ref{alg:determination-of-C-voronoi-neighbors}) we apply the
idea of simplex partitioning initially developed by Bundfuss and
D\"ur~\cite{Bundfuss2008a} to compute its copositive minimum
$\minC(B)$ and all vectors $\MinC(B)$ attaining it. Again we note that
this is not a polynomial time algorithm.

First we recall some facts and results from~\cite{Bundfuss2008a}.  A family
$\mathcal{P} = \{\Delta^1, \ldots, \Delta^m\}$ of simplices is called
a simplicial partitioning of the standard simplex $\Delta$ if
\[
\Delta = \bigcup_{i=1}^m \Delta^i \quad \text{with} \quad \interior(\Delta^i)
\cap \interior(\Delta^j) = \emptyset \quad \text{whenever } i \neq j.
\]
Let $v^k_1, \ldots, v^k_n$ be the vertices of simplex
$\Delta^k$. It is easy to verify that if a symmetric matrix
$B \in \mathcal{S}^n$ satisfies the strict inequalities
\begin{equation}
\label{eq:localcopcond}
(v^k_i)^{\sf T} B v^k_j > 0 \text{ for all } i, j = 1, \ldots, n,
\mbox{ and } k= 1, \ldots, m,
\end{equation}
then it lies in $\interior(\copn)$. Bundfuss and D\"ur \cite[Theorem
2]{Bundfuss2008a} proved the following converse: Suppose
$B \in \interior(\copn)$, then there exists an $\varepsilon > 0$ so that
for all finite simplex partitions
$\mathcal{P} = \{\Delta^1, \ldots, \Delta^m\}$ of $\Delta$, where the
diameter of every simplex $\Delta^k$ is at most $\varepsilon$, 
strict inequalities~\eqref{eq:localcopcond} hold. Here, the diameter of $\Delta^k$
is defined as $\max\{\|v^k_i - v^k_j\| : i, j = 1, \ldots, n\}$.

We assume now that $B \in \interior(\copn)$ and that we have a finite
simplex partition $\mathcal{P}$ so that~\eqref{eq:localcopcond}
holds. We furthermore assume that all the vertices $v^k_i$ have
rational coordinates. Such a simplex partition exists as shown by
Bundfuss and D\"ur \cite[Algorithm 2]{Bundfuss2008a}.

Each simplex $\Delta^k = \conv\{v^k_1, \ldots, v^k_n\}$ defines a
simplicial cone by $\cone\{v^k_1, \ldots, v^k_n\}$. From now on we
only work with the simplicial cones and not with the simplices any
more, so we may scale the rational $v^k_i$'s to have
integral coordinates. 

The goal is now to find all integer vectors $v$ in $\Delta^k$
which minimize $B[v]$. To do this we adapt the algorithm of Fincke and
Pohst \cite{Fincke1985a}, which solves the shortest lattice vector
problem. It is the corresponding problem for positive semidefinite
matrices. The adapted algorithm will solve the
following problem: Given a matrix $B \in \interior(\cpn)$ and a
simplicial cone, which is generated by integer vectors
$v_1, \ldots, v_n$ so that $v_i^{\sf T} B v_j \geq 0$ holds, and
given a positive constant $M$, find all integer vectors $v$
in the cone so that $B[v] \leq M$ holds. Then by reducing $M$
successively to $B[v]$, whenever such a non-trivial integer vector~$v$
is found,  we can find the copositive minimum of $B$ in the simplicial
cone, as well as all integer vectors attaining it.

The first step of the algorithm is to compute the Hermite normal form
of the matrix $V$ which contains the the vectors $v_1, \ldots, v_n$ as
it columns.  (see for example Kannan and Bachem~\cite{Kannan1979a} or
Schrijver~\cite{Schrijver1986a}, where it is shown that computing the
Hermite normal form can be done in polynomial time).  We find a
unimodular matrix $U \in \mathsf{GL}_n(\mathbb{Z})$ such that $UV = W$
holds, where $W$ is an upper triangular matrix with columns
$w_1, \ldots, w_n$ and coefficients $W_{i,j}$. Note that the diagonal
coefficients of $W$ are not zero since $W$ has full rank. Moreover,
denoting the matrix $(U^{-1})^{\sf T} B U^{-1}$ by $B'$ we have for
all $i, j$
\begin{equation}\label{eq:vBvpositive}
0\leq  v_i^{\sf T} B v_j 
= w_i^{\sf T} (U^{-1})^{\sf T} B U^{-1} w_j 
= w_i B' w_j 
,
\end{equation}
where the inequality is strict for whenever $i = j$.

We want to find all vectors
$v \in \cone\{v_1, \ldots, v_n\} \cap \mathbb{Z}^n$ so that
$B[v] \leq M$. In other words, the goal is to find all rational
coefficients $\alpha_1, \ldots, \alpha_n$ satisfying the following three
properties:
\begin{enumerate}
\item[(i)] $\alpha_1, \ldots, \alpha_n \geq 0$,
\item[(ii)] $\sum_{i=1}^n \alpha_i v_i \in \mathbb{Z}^n$,
\item[(iii)] $B\left[\sum_{i=1}^n \alpha_i v_i\right] \leq M$.
\end{enumerate}
Since matrix $U$ lies in $\mathsf{GL}_n(\mathbb{Z})$, a vector $\sum_{i=1}^n
\alpha_i v_i$ is integral if and only if $\sum_{i=1}^n
\alpha_i w_i$ is integral. Looking at the last vector componentwise we have
\[
\sum_{i=1}^n
\alpha_i w_i = \left(\sum_{j=1}^n \alpha_j W_{1,j}, \sum_{j=2}^n
  \alpha_j W_{2,j}, \ldots, \alpha_{n-1} W_{n-1,n-1} + \alpha_n
  W_{n-1,n}, \alpha_n W_{n,n} \right).
\]
We first consider the possible values of the last coefficient
$\alpha_n$, and then continue to other coefficients
$\alpha_{n-1}, \ldots, \alpha_1$, one by one via a backtracking
search. Conditions (i) and (ii) imply that
\[
\alpha_n \in \{k/W_{n,n} : k = 0, 1, 2, \ldots \}.
\]
Condition (iii) gives an upper bound for $\alpha_n$: Write $\alpha =
(\alpha_1, \ldots, \alpha_n)^{\sf T}$, then
\[
M \geq (V\alpha)^{\sf T} B V\alpha = \alpha^{\sf T} W^{\sf T} B' W \alpha
= B'\left[\sum_{i=1}^n \alpha_i w_i\right] \geq B'[\alpha_n w_n] = \alpha^2_n B'[w_n],
\]
where the last inequality follows from~\eqref{eq:vBvpositive}. Hence,
$\alpha_n \leq \sqrt{M/B'[w_n]}$ and so
\[
\alpha_n \in \left\{k/W_{n,n} : k = 0, 1, \ldots, \left\lfloor
    \sqrt{M/B'[w_n]} \right\rfloor W_{n,n}\right\}.
\]

Now suppose $\alpha_n$ is fixed. We want to compute all possible
values of the coefficient $\alpha_{n-1}$. Then the second but last
coefficient  $\alpha_{n-1} W_{n-1,n-1} + \alpha_n W_{n-1,n}$ should be
integral and $\alpha_{n-1}$ should be non-negative. Thus,
\[
\alpha_{n-1} \in \left\{ (k - \alpha_n W_{n-1,n})/W_{n-1,n-1} : k =
  \lceil \alpha_n W_{n-1,n} \rceil,  \lceil \alpha_n W_{n-1,n} \rceil + 1,
\ldots \right\}.
\]
Again we use condition (iii) to get an upper bound for $\alpha_{n-1}$:
\[
\begin{split}
& M \geq B'\left[\sum_{i=1}^n \alpha_i w_i\right] \geq B'[\alpha_{n-1}
w_{n-1} + \alpha_n w_n]\\
& \qquad  = \alpha_{n-1}^2 B'[w_{n-1}] +
2\alpha_{n-1}\alpha_n w_{n-1}^{\sf T} B' w_{n} + \alpha_n^2 B'[w_n],
\end{split}
\]
and solving the corresponding quadratic equation gives the desired
upper bound. 

Now suppose $\alpha_n$ and $\alpha_{n-1}$ are fixed. We want to compute all possible
values of the coefficient $\alpha_{n-2}$ and we can proceed inductively.

\subsection{Modifying the algorithm for general input}

In this section we discuss an adaption of
Algorithm~\ref{alg:new-voronoi-algorithm} for general symmetric
matrices $A$ as input. If $A$ is not in $\cpn$ then the procedure ends
with a separating witness matrix~$W$ if it terminates.  However, we
currently do not know if our
Procedure~\ref{alg:new-extended-procedure} always terminates in this
case (cf. Conjecture~\ref{conj:alg-works-outside}).

\begin{bigalg}   
\label{alg:new-extended-procedure}
\begin{minipage}{14cm}
\begin{flushleft}
\textbf{Input:} Rational symmetric matrix $A$\\[0.2cm] 
\textbf{Output:} If the procedure terminates: If $A \in \cpnrat$, then a rational
cp-factorization of $A$. If $A \not\in \cpn$ then a matrix $W \in
\cop^n$ with $\langle W, A \rangle < 0$.

\begin{enumerate}
\item[1.] Choose an initial $\cop$-perfect matrix  $P \in
  \mathcal{R}$; initialize $\mathcal{V}(P)$.\\

\item[2.] {\tt while} $A \not\in \mathcal{V}(P)$

\item[] \hspace*{0.5cm}(a)  {\tt if} $\langle P, A\rangle <0$ {\tt then} {\tt output}
  $A\not\in \cpn$ (with witness $W=P$)\\
  
\item[] \hspace*{0.5cm}(b) Determine a generator $R$ of an extreme ray of
  $(\MV(P))^{\ast}$ with $\langle  A,R \rangle < 0$.\\

\item[] \hspace*{0.5cm}(c) {\tt if} $R\in \copn$ {\tt then} {\tt output} 
  $A \not\in \cpn$  (with witness $W=R$)\\

\item[] \hspace*{0.5cm}(d) Use
  Algorithm~\ref{alg:determination-of-C-voronoi-neighbors} to
  determine the contiguous $\cop$-perfect matrix\\
  \hspace*{1cm} $N \leftarrow P + \lambda R$
  with $\lambda > 0$ and $\minC (N) = 1$. Compute $\mathcal{V}(N)$.

  \item[] \hspace*{0.5cm}(e) $P \leftarrow N$
    
\item[3.] Determine $\alpha_1, \ldots, \alpha_m \in \mathbb{Q}_{\geq 0}$ with
  $A = \sum_{i=1}^m \alpha_i v_iv_i^{\sf T}$ and output this rational
  cp-factorization.
\end{enumerate}
\end{flushleft}
\end{minipage}
\\[1ex]
{\sc Procedure \arabic{alg}.} Procedure for general input
\end{bigalg}

The difference between Algorithm~\ref{alg:new-voronoi-algorithm}
and Procedure~\ref{alg:new-extended-procedure} is in the new steps 
2(a) and 2(c). Here it is tested, whether or not we can already certify 
that the input matrix $A$ is not in $\cpn$.

\medskip

In \textbf{Step~2(a)} we check whether or not the current $\cop$-perfect
matrix~$P$ is already a separating witness. By this, the algorithm
subsequently constructs an outer approximation of the $\cpn$ cone:
$$
\cpn \subseteq
\{ Q\in \sn : 
\langle Q, B \rangle \geq 0 
\mbox{ for all constructed vertices $B$ of $\mathcal{R}$ } 
\}
$$
This procedure gives a tighter and tighter outer approximation of the
completely positive cone (cf. Theorem~\ref{thm:outer-approx}).

\medskip

In \textbf{Step~2(c)} it is checked whether or not $R$ is a separating
witness for $A$, that is, if not only $\langle A,R \rangle < 0$ but
also $R\in \copn$ holds. The copositivity test of $R$ can be realized as
explained in Section~\ref{ssec:checking_copositivity}.

\medskip

For the case of $A\not\in \cpn$, we do not know if it is possible that
Procedure~\ref{alg:new-extended-procedure} does not provide a
separating witness $W$ after finitely many iterations.  With a
suitably chosen rule in Step~2(b), however, we conjecture that the
computation finishes with a certificate:
 
\begin{conj}
\label{conj:alg-works-outside}
For $A\not\in \cpn$, Procedure~\ref{alg:new-extended-procedure} with a
suitable ``pivot rule'' in Step~2(b) ends after finitely many iterations
with a separating witness~$W$.
\end{conj}

We close this subsection with a few observations that can be made in
the remaining ``non-rational boundary cases'', that is, for
$A \in \bd \cpn \setminus \cpnrat$.  In this case,
Procedure~\ref{alg:new-extended-procedure} may not terminate after
finitely many steps, as shown in a $2$-dimensional example in the
following section.  Assuming there is an infinite sequence of vertices
$P^{(i)}$ of ${\mathcal R}$ constructed in
Procedure~\ref{alg:new-extended-procedure}, we know however at least
the following:

\begin{enumerate}
\item[(i)] The $\cop$-perfect matrix $P^{(i)}$ is in
  $ \{ B\in \copn : \langle P^{(i-1)}, A\rangle > \langle B,
  A\rangle \geq 0 \}.  $
\item[(ii)] The norms $\|P^{(i)} \|$ are unbounded. Otherwise --
  following the arguments in the proof of Lemma~2.4 in
  \cite{Dutour2017a} -- we could construct a convergent subsequence
  with limit $P \in {\mathcal R}$, for which we could then find a
  $u\in \R^n_{\geq 0}$ of norm $\|u\|=1$ with $P[u]=0$ (contradicting
  $P\in \interior(\copn)$).
\item[(iii)] $P^{(i)} / \|P^{(i)}\|$ contains a convergent
  subsequence with limit
  $P \in \{ X \in \sn : \langle X, A \rangle = 0\} $.  It can be shown
  that this $P$ is in $\bd \copn$.  Infinite sequences of vertices
  $P^{(i)}$ of ${\mathcal R}$ with such a limit $P$ exist.  For
  $n=2$ we give an example in Section~\ref{sec:2dexample}, in which
  $A$ is from the ``irrational boundary part''
  $(\bd \cpn ) \setminus \cpnrat$.
\end{enumerate}

\section{A $2$-dimensional example}
\label{sec:2dexample}

In this section we demonstrate how
Algorithm~\ref{alg:new-voronoi-algorithm} respectively
Procedure~\ref{alg:new-extended-procedure} 
works for $n = 2$. Thereby,
we discover a relation to beautiful classical results
in elementary number theory. In particular, we consider the case when
the input matrix~$A$ lies on the boundary of $\cptwo$, see Figure~\ref{fig:cp2}.

\begin{figure}
\unitlength1cm
\begin{picture}(12,5,12)
\put(2,0){\includegraphics[width=8cm]{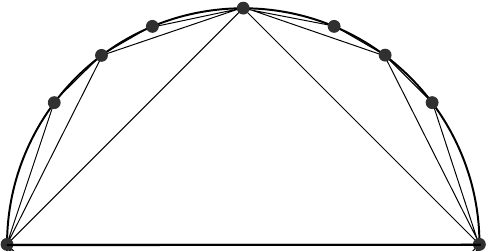}}
\put(5.7,4.3){$1,1$}
\put(4.0,4.0){$2,3$}
\put(2.9,3.4){$1,2$}
\put(2.1,2.5){$1,3$}
\put(1.3,0){$0,1$}
\put(10.1,0){$1,0$}
\put(9.3,2.5){$3,1$}
\put(8.5,3.4){$2,1$}
\put(7.4,4.0){$3,2$}
\end{picture}
\caption{Subdivision of $\cptwo$ by Voronoi cones $\mathcal{V}(P)$.
  Matrices $A = (a_{ij})$ are drawn with $2$-dimensional coordinates
  \[
(x,y) = \frac{1}{a_{11}+a_{22}} (a_{11}-a_{22}, a_{12}).
\]
  Integers $\alpha,\beta$ indicate that the shown point is on a ray
  spanned by the rank-$1$ matrix $A=v v^{\sf T}$ with
  $v=( \alpha , \beta )^{\sf T}$.  }
\label{fig:cp2}
\end{figure}

\subsection{Input on the boundary}

The boundary of $\cptwo$ splits into a part of diagonal matrices
\[
A = \begin{pmatrix} \alpha & 0 \\ 0 & \beta \end{pmatrix}
\quad \text{with} \quad \alpha, \beta \geq 0
\]
and into rank-$1$ matrices $A=xx^{\sf T}$.  In the first case,
Procedure~\ref{alg:new-extended-procedure} finishes already in its
first iteration, if we use $Q_{\mathsf{A}_2}$ as a starting perfect
matrix,\footnote{Strictly speaking we should use
  $\frac{1}{2}Q_{\mathsf{A}_2}$ here. If we use $Q_{\mathsf{A}_2}$
  instead, then the algorithm produces integral matrices and vertices
  of $2\mathcal{R}$.} where
\begin{equation}
\label{eq:start}
Q_{\mathsf{A}_2} = 
\begin{pmatrix} 2 & -1 \\ -1 & 2
\end{pmatrix}
\quad
\text{and}
\quad
\MinC (Q_{\mathsf{A}_2}) =
\left\{ 
\begin{pmatrix} 1 \\ 0 \end{pmatrix} ,
\begin{pmatrix} 0 \\ 1 \end{pmatrix} ,
\begin{pmatrix} 1 \\ 1 \end{pmatrix} 
\right\}.
\end{equation}

Let us consider the other boundary cases for $n=2$, where
$A = xx^{\sf T}$ is a rank-$1$ matrix.  Without loss of generality we
can assume that $x = (\alpha, 1)^{\sf T}$.  As we explain in the
following, Procedure~\ref{alg:new-extended-procedure} will terminate
after finitely many iterations with a $\cop$-perfect matrix $P$
satisfying $x \in \MinC P$ when $\alpha$ is rational.  For irrational
$\alpha$ the procedure will not terminate.

The first observation is that Procedure~\ref{alg:new-extended-procedure} 
subsequently replaces a $\cop$-perfect matrix $P$ by a contiguous
$\cop$-perfect matrix $N$ in a way that one of the three vectors in
$\MinC (P)$ is replaced by the sum of the remaining two. Let $P$
be a copositive matrix with
\begin{equation}
\label{eq:mincopbp}
\MinC P = \left\{ 
\begin{pmatrix} a \\ b \end{pmatrix} ,
\begin{pmatrix} c \\ d \end{pmatrix} ,
\begin{pmatrix} e \\ f \end{pmatrix} 
\right\}.
\end{equation}
It is known (see for example \cite [Section ``Determinants Determine
Edges'']{Hatcher2017a}) that
$\det\left(\begin{smallmatrix} a & c\\ b & d\end{smallmatrix}\right) =
\pm 1$ and we get a contiguous $\cop$-perfect matrix $N$ with
\[
\MinC N = \left\{
\begin{pmatrix} a \\ b \end{pmatrix} ,
\begin{pmatrix} c \\ d \end{pmatrix} ,
\begin{pmatrix} a+c  \\ b+d \end{pmatrix}
\right\}
\neq \MinC P
\]
by
\[
N = P + 4
\begin{pmatrix}
bd & -\frac{1}{2}(ad + bc)\\ 
-\frac{1}{2}(ad + bc) & ac
\end{pmatrix}.
\]
For instance, starting with $P = Q_{\mathsf{A}_2}$ as in~\eqref{eq:start}
\[
\begin{pmatrix} 1 \\ 0 \end{pmatrix} \text{  is replaced by }
\begin{pmatrix} 1 \\ 2 \end{pmatrix} \text{ if } \alpha < 1
\text{ yielding } 
N =
\begin{pmatrix}
6 & -3\\
-3 & 2
\end{pmatrix},
\]
or
\[
\begin{pmatrix} 0 \\ 1 \end{pmatrix} \text{ is replaced by }
\begin{pmatrix} 2 \\ 1 \end{pmatrix} \text{ if } \alpha >  1
\text{ yielding } 
N =
\begin{pmatrix}
2 & -3\\
-3 & 6
\end{pmatrix}.
\]
Note that for $\alpha =1$, Algorithm~\ref{alg:new-voronoi-algorithm}
also finishes already in the first iteration.  The way these vectors
are constructed corresponds to the way the famous \emph{Farey
  sequence} is obtained. This relation between the Farey
diagram/sequence and quadratic forms was first investigated in a
classical paper of Adolf Hurwitz \cite{Hurwitz1894a} in 1894 inspired by
a lecture of Felix Klein; see also the book by 
Hatcher~\cite{Hatcher2017a}, which contains the proofs.

For concreteness, let us choose $\alpha = \sqrt{2}$. Then $\MinC (P)$
is changed by replacing a suitable vector subsequently with
\[
\begin{pmatrix} 2 \\ 1 \end{pmatrix} ,
\begin{pmatrix} 3 \\ 2 \end{pmatrix} ,
\begin{pmatrix} 4 \\ 3 \end{pmatrix} ,
\begin{pmatrix} 7 \\ 5 \end{pmatrix} ,
\begin{pmatrix} 10 \\ 7 \end{pmatrix} ,
\begin{pmatrix} 17 \\ 12 \end{pmatrix} ,
\begin{pmatrix} 24 \\ 17 \end{pmatrix} ,
\begin{pmatrix} 41 \\ 29 \end{pmatrix} ,
\begin{pmatrix} 58 \\ 41 \end{pmatrix} ,
\begin{pmatrix} 99 \\ 70 \end{pmatrix} ,
\ldots
\]
Note that there is always a unique choice in 
Step~2(b) of Procedure~\ref{alg:new-extended-procedure} 
in case $A$ is a
$2\times 2$ rank-$1$ matrix. Note also that the vectors represent
fractions that converge to $\sqrt 2$.  Every second vector corresponds
to a convergent of the continued fraction expansion of $\sqrt{2}$: We have
\[
\sqrt{2} = 1 + \cfrac{1}{2 +\cfrac{1}{2+\cfrac{1}{2 + \cfrac{1}{2 +
        \cfrac{1}{2 + \ddots}}}}}
\]
and
\[
  3/2 = 1 + \frac{1}{2}, \;
  7/5 = 1 + \cfrac{1}{2 + \cfrac{1}{2}}, \; \ldots, \;
99/70 = 
1 + \cfrac{1}{2 +\cfrac{1}{2+\cfrac{1}{2 + \cfrac{1}{2 +
        \cfrac{1}{2}}}}}, \; \ldots
\]
The $\cop$-perfect matrix after ten iterations of the algorithm is
\[
P^{(10)}=
\begin{pmatrix}
4756 & -6726 \\
-6726 & 9512
\end{pmatrix}.
\]
It can be shown that the matrices $P^{(i)}$ converge to a multiple
of
\[
B =
\begin{pmatrix}
1 & -\sqrt{2} \\
-\sqrt{2} & 2
\end{pmatrix}
\text{ satisfying }
\langle A,B\rangle = 0 
\text{ and }
\langle X,B \rangle \geq 0 \text{ for all } X \in \cptwo.  
\]
However, every one of the infinitely many perfect matrices $P^{(i)}$
satisfies
\[
\langle X,P^{(i)}\rangle > 0 \text{ for all } X \in \cptwo.
\]

\subsection{Input outside}

In case the input matrix $A=(a_{ij})$ is outside of $\cptwo$ 
we distinguish two cases using the starting $\cop$-perfect matrix $Q_{\mathsf{A}_2}$:
If $a_{12} = a_{21} < 0$ then 
Procedure~\ref{alg:new-extended-procedure} finishes already in its
first iteration (in Step~2(c)) with a separating witness
$$
W = R = 
\begin{pmatrix}
0 & 1 \\
1 & 0
\end{pmatrix}
.
$$
If $a_{12} = a_{21} \geq 0$,
Procedure~\ref{alg:new-extended-procedure} 
terminates after finitely many iterations (in Step~2(a))  
with a separating $\cop$-perfect witness matrix $W = P$.

We additionally note that it is a special feature of the $n=2$ case
that we can conclude that the input matrix $A$ is outside of $\cptwo$
if we have a choice between two possible $R$ with
$\langle A,R \rangle < 0$ in Step~2(b) of
Procedure~\ref{alg:new-extended-procedure}.

\subsection{Integral input}

Laffey and \v{S}imgoc \cite{Laffey2018a} showed that every integral
matrix $A \in \cptwo$ possesses an integral cp-factorization. This can
also be seen as follows: If $P$ is a copositive matrix with 
$\MinC P$ as in \eqref{eq:mincopbp}
then the matrices
\[
\begin{pmatrix}
a\\
b
\end{pmatrix}
\begin{pmatrix}
a\\
b
\end{pmatrix}^{\sf T},
\begin{pmatrix}
c\\
d
\end{pmatrix}
\begin{pmatrix}
c\\
d
\end{pmatrix}^{\sf T},
\begin{pmatrix}
e\\
f
\end{pmatrix}
\begin{pmatrix}
e\\
f
\end{pmatrix}^{\sf T}
\]
form a Hilbert basis of the convex cone which they generate. This
means that every integral matrix in this cone is an integral
combination of the three matrices above. To show this, one immediately
verifies this fact in the special case of $P = Q_{{\sf A}_2}$. 
Then all the other cones are equivalent by conjugating with a matrix
in $\mathsf{GL}_2(\mathbb{Z})$.

\section{Computational Experiments}
\label{sec:experiments}

We implemented our algorithm. The source code, written in C++, is
available on {\tt GitHub} \cite{github}.  In this section we report on
the performance on several examples, most of them previously discussed
in the literature.  Generally, the running time of the procedure is
hard to predict.  The number of necessary iterations in
Algorithm~\ref{alg:new-voronoi-algorithm} respectively
Procedure~\ref{alg:new-extended-procedure} drastically varies in the
considered examples.  Most of the computational time is taken by the
computation of the copositive minimum as described in
Section~\ref{ssec:computing_minC}.

\subsection{Matrices in the interior}

For matrices in the interior of the completely positive cone, our
algorithm terminates with a certificate in form of a cp-factorization.
Note that in \cite{DuerStill2008} and in 
\cite{Dickinson2010a} characterizations of matrices in
the interior of the completely positive cone are given.
For example, we have that $A \in \interior(\cpn)$ if and only if $A$
has a factorization $A = BB^{\sf T}$ with $B > 0$ and $\rank B = n$.

The matrix 
\[
\begin{pmatrix}
 6  &  7  &  8  &  9  &  10  &  11\\
 7  &  9  &  10  &  11  &  12  &  13\\
 8  &  10  &  12  &  13  &  14  &  15\\
 9  &  11  &  13  &  15  &  16  &  17\\
 10  &  12  &  14  &  16  &  18  &  19\\
 11  &  13  &  15  &  17  &  19  &  21
\end{pmatrix}
\] 
for example lies in the interior of $\mathcal{CP}_6$, 
as it has a cp-factorization with vectors
$(1, 1, 1, 1, 1, 1)$, 
$(1, 1, 1, 1, 1, 2)$, 
$(1, 1, 1, 1, 2, 2)$, 
$(1, 1, 1, 2, 2, 2)$, 
$(1, 1, 2, 2, 2, 2)$ and
$(1, 2, 2, 2, 2, 2)$. 
It is found after 8~iterations of our algorithm.

\subsection{Matrices on the boundary}

For matrices in $\cpnrat$ there exists a cp-factorization by
definition. However, on the boundary of the cone these are often
difficult to find.

The following example is from \cite{Groetzner2018a} and lies in the boundary of $\cprat^5$: 
\[
\begin{pmatrix}
8 & 5 & 1 & 1 & 5\\
5 & 8 & 5 & 1 & 1\\
1 & 5 & 8 & 5 & 1\\
1 & 1 & 5 & 8 & 5\\
5 & 1 & 1 & 5 & 8\\
\end{pmatrix}
\]
Starting from $Q_{{\mathsf{A}_5}}$ our algorithm needs 5~iterations to
find the cp-factorization with the ten vectors $(0, 0, 0, 1, 1)$,
$(0, 0, 1, 1, 0)$, $(0, 0, 1, 2, 1)$, $(0, 1, 1, 0, 0)$,
$(0, 1, 2, 1, 0)$, $(1, 0, 0, 0, 1)$, $(1, 0, 0, 1, 2)$,
$(1, 1, 0, 0, 0)$, $(1, 2, 1, 0, 0)$ and $(2, 1, 0, 0, 1)$.

While the above example can be solved within seconds on a standard
computer, the matrix 
\[
A = 
\begin{pmatrix}
41  &  43  &  80   &  56   &  50\\
43  &  62  &  89   &  78   &  51\\
80  &  89  &  162  &  120  &  93\\
56  &  78  &  120  &  104  &  62\\
50  &  51  &  93   &  62   &  65
\end{pmatrix}
\] 
from Example~7.2 in \cite{Groetzner2018a} took roughly 10~days and
70~iterations to find a factorization with only three vectors
$(3,5,8,8,2)$, $(4,1,7,2,5)$ and $(4,6,7,6,6)$.  The second algorithm
suggested in~\cite{Groetzner2018a} found the following approximate 
cp-factorization in 0.018 seconds
\[
A = \tilde{B}\tilde{B}^{\sf T}, \; \text{ with } \;
\tilde{B} =
\begin{pmatrix}
0.0000 & 3.3148 & 4.3615 & 3.3150 & 0.0000\\
0.0000 & 0.7261 & 4.3485 & 6.5241 & 0.0000\\
0.0000 & 4.5242 & 9.9675 & 6.4947 & 0.0000\\
0.0000 & 0.1361 & 7.4192 & 6.9955 & 0.0000\\
0.0000 & 5.3301 & 3.8960 & 4.6272 & 0.0000
\end{pmatrix}.
\]

We also considered the following family of completely positive
$(n+m)\times (n+m)$ matrices, generalizing the family of examples
considered in \cite{Jarre2009a}: The matrices
\[
\begin{pmatrix}
n \Id_m     & J_{m,n}\\
J_{n,m}    & m \Id_n
\end{pmatrix},
\]
with $J_{\cdot,\cdot}$ denoting an all-ones matrix of suitable size,
are known to have $cp$-rank $nm$, that is, 
they have a cp-factorization with $nm$ vectors, but not with less.
These factorizations are found by our algorithm
with starting $\cop$-perfect matrix $Q_{{\mathsf{A}_{m+n}}}$
for all $n,m\leq 3$ in less than 6~iterations.

\subsection{Matrices that are not completely positive}

For matrices that are not completely positive, our algorithm can find
a certificate in form of a witness matrix that is copositive.

The following example is taken from \cite[Example 6.2]{Nie2014a}. 
\[
A=
\begin{pmatrix}
1 & 1 & 0 & 0 & 1\\
1 & 2 & 1 & 0 & 0\\
0 & 1 & 2 & 1 & 0\\
0 & 0 & 1 & 2 & 1\\
1 & 0 & 0 & 1 & 6
\end{pmatrix}
\]
is positive semidefinite, but not completely positive.
Starting from $Q_{{\mathsf{A}_5}}$ our algorithm needs $18$~iterations to
find the copositive witness matrix
\[
B =
\begin{pmatrix}
 363/5 &-2126/35 & 2879/70 & 608/21 & -4519/210\\
 -2126/35 & 1787/35 & -347/10 & 1025/42 & 253/14\\
 2879/70 & -347/10 & 829/35 & -1748/105 & 371/30\\
 608/21 & 1025/42 & -1748/105 & 1237/105 & -601/70\\
 -4519/210 & 253/14 & 371/30 & -601/70 & 671/105
\end{pmatrix}
\]
with $\langle A, B \rangle = -2/5$, verifying $A \not\in \mathcal{CP}_5$.

\section*{Acknowledgement}

We like to thank Jeff Lagarias for a helpful discussion about Farey
sequences and Renaud Coulangeon for pointing out the link to the work
of Opgenorth. We also like to thank Valentin Dannenberg, 
Peter Dickinson and Veit Elser for helpful suggestions. 
Finally, we are very grateful to the referees for
carefully reading our paper and for their detailed comments. This
helped us to improve the presentation of our paper substantially.

\smallskip

This project has received funding from the European Union's Horizon
2020 research and innovation programme under the Marie
Sk\l{}odowska-Curie agreement No 764759
and it is based upon work supported by the National
Science Foundation under Grant No. DMS-1439786 while the authors were 
in residence at the Institute for Computational and Experimental
Research in Mathematics in Providence, RI, 
during the Spring 2018 semester.
Moreover, the first author was supported by the Humboldt foundation and the
  third author is partially supported by the
SFB/TRR 191 ``Symplectic Structures in Geometry,
Algebra and Dynamics'', funded by the DFG.

\end{document}